\documentclass[reqno]{amsart}
\usepackage{mathtools}
\usepackage{amssymb}
\usepackage{leftidx}
\usepackage{extarrows}
\newtheorem{theorem}{Theorem}[section]
\newtheorem{lemma}[theorem]{Lemma}
\newtheorem{prop}[theorem]{Proposition}

\theoremstyle{definition}

\newtheorem{cor}[theorem]{Corollary}

\theoremstyle{remark}
\newtheorem{remark}[theorem]{Remark}

\numberwithin{equation}{section}

\newcommand{\abs}[1]{\lvert#1\rvert}

\renewcommand{\leq}{\leqslant}
\renewcommand{\geq}{\geqslant}

\newcommand{\la}{\lambda}

\newcommand{\Z}{\mathbb{Z}}

\def \bp{\begin{prop}\label}
\def \ep{\end{prop}}
\def \bt{\begin{theorem}\label}
\def \et{\end{theorem}}

\begin{document}

\title{Graded characters, Demazure multiplicities, and Chebyshev polynomials}

\author{Rekha Biswal}

\address{
School of Mathematical Sciences, National Institute of Science Education and Research, Bhubaneswar, an OCC
of Homi Bhabha National Institute(HBNI), P. O. Jatni, Khurda, 752050, Odisha, India
}

\email{rekhabiswal27@gmail.com, rekha@niser.ac.in}

\date{}

\begin{abstract}
In this paper, we study numerical multiplicities of Demazure modules in the excellent filtration of $\mathfrak{sl}_2[t]$-modules $V(\xi)$, where $V(\xi)$ denotes the fusion product associated to a partition $\xi$. We express generating functions for the numerical multiplicities of level $m$ Demazure modules in excellent filtrations of $V(\xi)$ in terms of quotients of Chebyshev polynomials, thereby generalizing earlier results for fat hook partitions.

We also revisit the graded multiplicities of irreducible $\mathfrak{sl}_2$-modules in $V(\xi)$ and provide a new and self-contained proof of their description in terms of cocharge Kostka--Foulkes polynomials. While this connection has been established in earlier works, our approach is elementary and relies only on recursive structures arising from short exact sequences of fusion products. As a consequence, we obtain a direct and self-contained derivation of the graded characters of $V(\xi)$ in terms of Hall--Littlewood polynomials with coefficients given by cocharge Kostka--Foulkes polynomials.

\end{abstract}

\maketitle

\section{Introduction}

The representation theory of current algebras, and in particular that of $\mathfrak{sl}_2[t]$, has
been the subject of extensive study in recent years due to its deep and fruitful connections with
algebraic combinatorics, symmetric functions, and the theory of affine Lie algebras. A central
theme in this area is the study of finite-dimensional graded $\mathfrak{sl}_2[t]$-modules and, more
specifically, the determination of their graded characters and decomposition into irreducible
$\mathfrak{sl}_2$-modules.

An important class of such modules is given by fusion products, introduced and systematically
studied by Chari and Venkatesh \cite{CV}. These modules are parametrized by partitions and form a rich
family that includes several well-known representations, such as local Weyl modules and Demazure
modules. Their defining relations encode subtle combinatorial data, and their structure reflects
intricate interactions between representation theory and symmetric function theory.

A fundamental problem is to describe the graded multiplicities of irreducible
$\mathfrak{sl}_2$-modules occurring in these fusion products. Despite significant progress in
understanding their structure, explicit and conceptual descriptions of these graded multiplicities
remain of considerable interest.

On the combinatorial side, Kostka--Foulkes polynomials and modified Hall--Little-wood polynomials occupy a
central position in the theory of symmetric functions. These polynomials encode refined information
about weight multiplicities and carry deep connections to representation theory, geometry, and
combinatorics. In particular, the cocharge Kostka--Foulkes polynomials provide a natural
$q$-analogue of weight multiplicities and encode graded multiplicities of
irreducible modules in various representation-theoretic constructions.

Our first main result concerns the study of excellent filtrations (also known as Demazure flags)
of fusion products. It is known that, under suitable conditions on the partition $\xi$, the module
$V(\xi)$ admits a filtration whose successive quotients are Demazure modules of a fixed level.
The graded multiplicities of these Demazure modules, and their associated generating functions,
encode important structural information about $V(\xi)$.

We study these multiplicities from a combinatorial perspective. In particular, we study the
generating functions associated to Demazure multiplicities and show that, at the level of numerical
multiplicities (i.e., upon specialization $q = 1$), these generating functions admit a simple and
explicit description in terms of Chebyshev polynomials of the second kind. A connection of this
type was previously observed in \cite{BCSV} in the special case of local Weyl modules. Our results
extend this connection to the more general class of modules $V(\xi)$, thereby providing a uniform
description valid for arbitrary partitions $\xi$.

Our second main result establishes a direct and explicit link between graded multiplicities of
irreducible $\mathfrak{sl}_2$-modules in $V(\xi)$ and cocharge Kostka--Foulkes polynomials. More
precisely, we show that these multiplicities are given by suitably normalized cocharge
Kostka--Foulkes polynomials.

Connections between graded characters of such modules and Kostka--Foulkes
polynomials have already been established in the literature, notably through
their interpretation as one-dimensional sums arising from crystal bases and
their realization as graded multiplicities in fusion products of Demazure modules;
see, for example, the work of Kedem and Naoi \cite{Ke,Naoi,Naoi1}. In particular,
these results imply the graded character formula for the modules $V(\xi)$.

In contrast, the approach in the present paper is entirely self-contained and
relies only on elementary and recursive properties of the modules $V(\xi)$,
avoiding the use of crystal-theoretic or affine Lie algebra techniques. Our
arguments are based on short exact sequences and explicit combinatorial
recursions, leading to a direct and conceptually transparent derivation of the
graded character formula, which we believe is of independent interest.

The methods developed here combine recursive structures arising from short exact
sequences of fusion products with combinatorial properties of Kostka--Foulkes
polynomials and generating function techniques. In particular, the recursive
description of fusion products plays a key role in relating graded multiplicities
to cocharge Kostka polynomials while analogous recursive structures for generating functions lead naturally to Chebyshev polynomials.

The paper is organized as follows. In Section 2 we recall the notion of excellent filtrations,
study the numerical multiplicities in the excellent filtration of the modules $V(\xi)$ and prove
their connection with Chebyshev polynomials. In Section 3 we prove our main result expressing
graded multiplicities of irreducible modules in $V(\xi)$ in terms of cocharge Kostka--Foulkes
polynomials by establishing a recursive formula for these polynomials adapted to our setting.

\section{Excellent filtrations and graded multiplicities}

In this section we recall the notion of excellent filtrations (Demazure flags) \cite{CSSW, BCSV} and set up
the notation needed in the rest of the paper.

\subsection*{2.1.}
Throughout this paper we denote by $\mathbb{C}$ the field of complex numbers and by $\mathbb{Z}$ (resp.
$\mathbb{Z}_+, \mathbb{N}$) the subset of integers (resp. non–negative, positive integers).
Let $\mathfrak{sl}_2[t] \cong \mathfrak{sl}_2 \otimes \mathbb{C}[t]$ be the Lie algebra of two by two matrices of trace zero with entries
in $\mathbb{C}[t]$. The degree grading of $\mathbb{C}[t]$ defines
a natural grading on $\mathfrak{sl}_2[t]$. A finite-dimensional $\mathbb{Z}$-graded $\mathfrak{sl}_2[t]$–module is a $\mathbb{Z}$–graded vector
space admitting a compatible graded action of $\mathfrak{sl}_2[t]$:
\[
V = \bigoplus_{k\in \mathbb{Z}} V[k], \quad (a \otimes t^r)V[k] \subset V[k+r], \; a \in \mathfrak{sl}_2, \; r \in \mathbb{Z}_+.
\]

Given a $\mathbb{Z}$-graded $\mathfrak{sl}_2[t]$-module $V$, we denote by $\tau_p^{*}V$
the graded module obtained by shifting the grading, where
\[
\tau_p^{*}V[k] = V[k-p].
\]

\subsection*{2.2.}
The category of finite-dimensional $\mathbb{Z}$–graded $\mathfrak{sl}_2[t]$–modules was the central subject
of many recent papers (see, for example \cite{BCSV, CSSW, CV, FL, KL}). There is a well–known family of
objects in that category which is of particular interest, namely the subclass of fusion products.
We recall their description in terms of generators and relations from \cite[Sec.~6]{CV} ; for a more
traditional definition we refer the reader to \cite[Sec.~5.5]{FL}. Let $x,h,y$ be the standard basis of $\mathfrak{sl}_2$, given by
\[
x = \begin{pmatrix} 0 & 1 \\ 0 & 0 \end{pmatrix}, \quad
h = \begin{pmatrix} 1 & 0 \\ 0 & -1 \end{pmatrix}, \quad
y = \begin{pmatrix} 0 & 0 \\ 1 & 0 \end{pmatrix},
\]
which satisfy the relations
\[
[x,y] = h, \quad [h,x] = 2x, \quad [h,y] = -2y.
\]   Set $u^{(r)} := \frac{1}{r!}u^r$ for $u \in \mathfrak{sl}_2[t], r \in \mathbb{Z}_+$. For a partition $\xi=(\xi_1 \ge \xi_2 \ge \cdots \ge \xi_\ell > 0)$, we set $|\xi|=\sum_{i=1}^{\ell}\xi_i$.

The fusion product associated to $\xi$ is the $\mathfrak{sl}_2[t]$–module $V(\xi)$ generated by an element $v_\xi$ with
defining relations:
\[
(x \otimes \mathbb{C}[t])v_\xi = 0, \quad (h \otimes f)v_\xi = |\xi| f(0)v_\xi, \quad (y \otimes 1)^{|\xi|+1}v_\xi = 0, \tag{2.1}
\]
\[
(x \otimes t)^{(p)}(y \otimes 1)^{(r+p)}v_\xi = 0, \quad r, p \in \mathbb{N}, \; r + p \ge 1 + rk + \sum_{j \ge k+1} \xi_j \text{ for some } k \in \mathbb{N}. \tag{2.2}
\]

It turns out that many other well–known families of representations belong to the class of fusion
products. For example, Demazure modules occur in irreducible integrable representations of
the affine Lie algebra $\widehat{\mathfrak{sl}}_2$ and are parametrized by tuples $(m,n) \in \mathbb{N} \times \mathbb{Z}_+$, where the integer
$m$ is called the level. We denote such a module by $D(m,n)$. If $n_0, n_1 \in \mathbb{Z}_+$ are such that
$n_0 < m$ and $n = n_1 m + n_0$, then the fusion product $V(\xi(m,n))$ associated to the partition
$\xi(m,n) := (m^{n_1}, n_0)$ is isomorphic to $D(m,n)$ (see \cite[Theorem~2]{CV}). Hence Demazure modules
can be categorized into the family of fusion products, but the class of fusion products is
generically much bigger. Nevertheless, there is a beautiful result saying that fusion products
admit a filtration by Demazure modules under a suitable condition on the partition. The
following proposition was proved in \cite[Theorem~3.3]{CSSW}.

\medskip

\noindent
\textbf{Proposition.}
Let $\xi = (\xi_1 \ge \xi_2 \ge \cdots \ge \xi_\ell>0)$ a partition and $m \in \mathbb{N}$. The module $V(\xi)$ admits
a filtration of level $m$, i.e., there exists a decreasing sequence of graded submodules
\[
0 = V_0 \subset V_1 \subset \cdots \subset V_{k-1} \subset V_k = V(\xi)
\]
such that
\[
V_i/V_{i-1} \cong \tau^*_{p_i} D(m,n_i), \quad (p_i, n_i) \in \mathbb{Z}_+ \times \mathbb{Z}_+, \; 1 \le i \le k
\]
if and only if $m \ge \xi_1$. $\hfill \square$

\medskip

These types of filtrations are called excellent filtrations (see \cite{BK, BCSW}) but are also known in the literature as level $m$–Demazure flags (see \cite{CSSW, BCK}). The aim of this section is to give a combinatorial
formula for generating series of the numerical multiplicities in excellent filtrations.

\medskip

\noindent

We will implicitly assume in the rest of the paper that $m \ge \xi_1$ whenever we talk about level
$m$–Demazure flags of $V(\xi)$.

\subsection*{2.3.}
The number of times a particular level $m$–Demazure module $D(m,n)$ appears
in a level $m$–flag is independent of the choice of the flag \cite{CSSW}. We encode these multiplicities in the
polynomial
\[
V_{\xi \to m}^n(q) := \sum_{p \ge 0} [V(\xi) : \tau^*_p D(m,n)] q^p, \tag{2.3}
\]
where
\[
[V(\xi) : \tau^*_p D(m,n)] = \#\{1 \le i \le k : V_i/V_{i-1} \cong \tau^*_p D(m,n)\}.
\]

\subsection*{2.4.}
For a partition $\xi=(\xi_1 \ge \xi_2 \ge \cdots \ge \xi_\ell>0)$ we define an associated pair of partitions $\xi^\pm$ as follows. If $\ell(\xi) = 1$, then
$\xi^+ = \xi$ and $\xi^-$ is the empty partition. Otherwise set $\xi^- = (\xi_1 \ge \xi_2 \ge \cdots \ge \xi_{\ell-2} \ge \xi_{\ell-1} - \xi_\ell)$
and $\xi^+$ the unique partition associated to the tuple $(\xi_1, \xi_2, \ldots, \xi_{\ell-1} + 1, \xi_\ell - 1)$.
We record a lemma which will be needed later; for a proof see \cite[Lem.~3.8 and Eq.~(3.5)]{CSSW}.

\medskip

\noindent
\begin{lemma}\label{recursion}
Let $\xi = (\xi_1 \ge \xi_2 \ge \cdots \ge \xi_\ell > 0)$ be a partition and set $\xi' = (\xi_2 \ge \xi_3 \ge \cdots \ge \xi_\ell > 0)$. We have

(1) $V_{\xi \to \xi_1}^n(q) = q^{(|\xi|-n)/2} V_{\xi' \to \xi_1}^{n-\xi_1}(q)$ if $n \geq \xi_1$.

(2) $V_{\xi \to m}^n(q) = V_{\xi^+ \to m}^n(q) + q^{(\ell-1)\xi_\ell} V_{\xi^- \to m}^n(q)$.

$\hfill \square$
\end{lemma}
\medskip

\noindent
\textbf{Remark.}
We see that Lemma \ref{recursion} has the advantage that one can derive recursive formulas
for $V_{\xi \to m}^{n}(q)$ (see for example \cite[Prop.~2.3]{BCSV}) but closed formulas were established only in
the case of $\xi = (1^s) ,m = 2$ \cite[Thm.~3.3]{CSSW} and $\xi = (1^s) ,m = 3$ for all $s \in \mathbb{N}$ \cite[Sec.~1.6]{BCSV}. Combinatorial formulae for the coefficients of $q$ in $V_{\xi \to m}^{n}(q)$ are also established in \cite{BK} for the case where $\xi$ is a fat hook partition in terms of admissible Dyck paths with comajor index statistics. In the theorem below, we express the generating function for the numerical multiplicities ($q=1$) of level $m$-Demazure modules in the excellent filtration of modules $V(\xi)$ for an arbitrary partition $\xi$ in terms of quotients of Chebyshev polynomials which was earlier established in \cite{BCSV} only for the case where $\xi=(1^s)$ for some $s \in \mathbb{N}$ .

\subsection*{2.5.}
The numerical multiplicities have been intensively studied for the partition $\xi$ where $\xi_i = 1$,
$1 \le i \le \ell$. It turns out that the generating function 

\[
A_{1 \to m}^{\,n}(x, q)
= \sum_{r \geq 0}
V_{(1^{n+2r}) \to m}^{\,n}(q)\, x^r
\]
 evaluated at $q = 1$ is a
rational function which can be expressed in terms of Chebyshev polynomials \cite{BCSV}. We shall explain
this connection in more detail and give a more general version of this result. The Chebyshev
polynomials of the second kind are defined by the recurrence relation
\[
U_n(x) = 2xU_{n-1}(x) - U_{n-2}(x), \quad n \ge 2,
\]
with initial data $U_0(x) = 1$ and $U_1(x) = 2x$. Define polynomials
\[
p_n(x) := x^{n/2} U_n((2\sqrt{x})^{-1}) = \sum_{s=0}^{\lfloor n/2 \rfloor} (-1)^s
\begin{bmatrix}
n - s \\
s
\end{bmatrix}
x^s.
\]
Then $\{p_n(x)\}_{n\ge 0}$ is defined by the recurrence
\[
p_{n+1}(x)=p_n(x)-x p_{n-1}(x), \qquad p_0(x)=p_1(x)=1.
\]
Further, set
\[
p_\xi(x) = \prod_{i=1}^\ell p_{\xi_i}(x).
\]

\begin{lemma}\label{recursionchebyshev}

For all integers $i \ge j \ge 1$, we have
\[
p_i(x)\,p_j(x)=p_{i+1}(x)\,p_{j-1}(x)+x^j p_{i-j}(x).
\]
\end{lemma}

\begin{proof}
We proceed by induction on $j$.

\textbf{Base case ($j=1$):} We verify
\[
p_i p_1 = p_{i+1}p_0 + x p_{i-1}.
\]
Since $p_0=p_1=1$, this becomes
\[
p_i = p_{i+1} + x p_{i-1},
\]
which is exactly the recurrence relation rewritten from
\[
p_{i+1}=p_i-x p_{i-1}.
\]
Thus the statement holds for $j=1$.

\textbf{Inductive step:} Assume that for some $j \ge 1$ and all $i \ge j$,
\[
p_i p_j = p_{i+1}p_{j-1} + x^j p_{i-j}.
\]
We prove the result for $j+1$.

Using the recurrence $p_{j+1}=p_j - x p_{j-1}$, we get
\[
p_i p_{j+1} = p_i p_j - x p_i p_{j-1}.
\]

Applying the induction hypothesis to both terms yields
\[
p_i p_j = p_{i+1}p_{j-1} + x^j p_{i-j},
\]
\[
p_i p_{j-1} = p_{i+1}p_{j-2} + x^{j-1} p_{i-j+1}.
\]

Substituting these expressions,
\[
p_i p_{j+1}
= p_{i+1}p_{j-1} + x^j p_{i-j}
- x\big(p_{i+1}p_{j-2} + x^{j-1}p_{i-j+1}\big).
\]

Rearranging,
\[
p_i p_{j+1}
= p_{i+1}(p_{j-1} - x p_{j-2}) + x^j (p_{i-j} - p_{i-j+1}).
\]

Using the recurrence relations
\[
p_j = p_{j-1} - x p_{j-2}, \qquad
p_{i-j} - p_{i-j+1} = x p_{i-j-1},
\]
we obtain
\[
p_i p_{j+1}
= p_{i+1}p_j + x^{j+1} p_{i-j-1}.
\]

This completes the induction.
\end{proof}

For a series $f(x) = \sum_{k \ge 0} a_k x^k$ with complex coefficients we denote by $f(x)[k]$ its $k$–th coefficient $a_k$. Then we have following identity:
\[
p_\xi(x)[k] = p_{\xi^+}(x)[k] + p_{\xi^-}(x)[k - \xi_\ell]. \tag{2.6}
\] by applying Lemma \ref{recursionchebyshev} to the last two factors in the product defining $p_\xi(x)$ .
Hence 

\begin{equation}\tag{2.7}
p_{\xi^+}(x) + x^{\xi_\ell} p_{\xi^-}(x) = p_\xi(x).
\end{equation}

\medskip

\noindent
\begin{theorem}\label{mainthm}
Let $\xi = (\xi_1 \ge \xi_2 \ge \cdots \ge \xi_\ell > 0)$ be a partition, let $m \in \mathbb{N}$ with $m \ge \xi_1$, and let $n \in \mathbb{Z}_+$. Assume that $|\xi|-n \in 2\mathbb{Z}$. Write $n = n_1 m + n_0$ with $n_0, n_1 \in \mathbb{Z}_+$ and $0 \le n_0 < m$. Then the numerical multiplicity is given by
\[
V_{\xi \to m}^n(1) =
\Big(\frac{p_{m-n_0-1}(x)\,p_\xi(x)}{p_m(x)^{n_1+1}}\Big)
\left[\frac{\abs{\xi} - n}{2}\right].
\]
\end{theorem}

\medskip

\noindent\textit{Proof.}
We prove the statement by induction on the following partial order on partitions. Given partitions $\xi$ and $\xi'$ with $\ell$ and $\ell'$ parts, respectively, we write $\xi > \xi'$ if either $\ell > \ell'$, or $\ell = \ell'$ and $\xi_\ell > \xi'_\ell$. In particular, for the partitions $\xi^+$ and $\xi^-$ appearing in Lemma~\ref{recursion}, one has $\xi > \xi^+$ and $\xi > \xi^-$.

\medskip

\noindent\emph{Base case.} When $\ell = 1$, say $\xi = (r)$,  $V(\xi)$ is just the irreducible module $V(r)$ of $\mathfrak{sl}_2$, and one has $[V(\xi):D(m,n)]=\delta_{n,r}$ for every $m \geq r$. This agrees with the corresponding coefficients of the Chebyshev quotients, and hence verifies the claim for all minimal elements in the above partial order.
\medskip

\noindent\emph{Inductive step.} Assume that the formula holds for all partitions $\xi' < \xi$. We prove it for $\xi$.

By Lemma~\ref{recursion}, we have
\begin{align*}
V_{\xi \to m}^n(1)
&= V_{\xi^+ \to m}^n(1) + V_{\xi^- \to m}^n(1).
\end{align*}
Since $\xi^+, \xi^- < \xi$, the induction hypothesis applies to both terms, and hence
\begin{align*}
V_{\xi^+ \to m}^n(1)
&= \Big(\frac{p_{m-n_0-1}(x)\,p_{\xi^+}(x)}{p_m(x)^{n_1+1}}\Big)
\left[\frac{\abs{\xi} - n}{2}\right], \\
V_{\xi^- \to m}^n(1)
&= \Big(\frac{p_{m-n_0-1}(x)\,p_{\xi^-}(x)}{p_m(x)^{n_1+1}}\Big)
\left[\frac{\abs{\xi} - n}{2} - \xi_\ell\right].
\end{align*}
Combining these, we obtain
\begin{align*}
V_{\xi \to m}^n(1)
&= \Big(\frac{p_{m-n_0-1}(x)}{p_m(x)^{n_1+1}}\Big)
\Bigg(
 p_{\xi^+}(x) \left[\frac{\abs{\xi} - n}{2}\right]
 + p_{\xi^-}(x) \left[\frac{\abs{\xi} - n}{2} - \xi_\ell\right]
\Bigg).
\end{align*}
Using the identity \emph{(2.7)} relating $p_{\xi}(x)$ with $p_{\xi^+}(x)$ and $p_{\xi^-}(x)$, namely
\[
p_{\xi}(x) = p_{\xi^+}(x) + x^{\xi_\ell} p_{\xi^-}(x),
\]
and the corresponding shift property of coefficients, we deduce that the expression in parentheses equals
\[
p_{\xi}(x)\left[\frac{\abs{\xi} - n}{2}\right].
\]
Therefore,
\[
V_{\xi \to m}^n(1)
= \Big(\frac{p_{m-n_0-1}(x)\,p_{\xi}(x)}{p_m(x)^{n_1+1}}\Big)
\left[\frac{\abs{\xi} - n}{2}\right],
\]
as required. This completes the induction.
\hfill $\square$



\begin{remark}
It is a natural and interesting problem to seek a combinatorial interpretation of the coefficients
of $x$ appearing in the expression of Theorem \ref{mainthm}. In particular, it would be desirable to describe
these coefficients in terms of suitable statistics on Dyck paths or related lattice path models.

Such a description is known in the special case when $\xi$ is a fat hook partition, where graded
multiplicities admit a combinatorial interpretation in terms of admissible Dyck paths with comajor
index statistics (see \cite{BK}). It would be very interesting to develop analogous combinatorial
models in the general setting considered here.

We also note an intriguing feature of Theorem \ref{mainthm}: although the polynomials $p_n(x)$ (and hence the
numerator appearing in the quotient of Chebyshev polynomials) involve alternating signs and thus
have coefficients which are not manifestly non-negative, the theorem implies that the specific
coefficient
\[
[x^{\frac{|\xi|-n}{2}}]
\]
of the resulting quotient, which computes the numerical multiplicity, is a non-negative integer.
In general, the quotient itself need not have all coefficients non-negative, and initial coefficients
may even be negative. It would be desirable to understand this distinguished-coefficient positivity
directly at a combinatorial level, and more generally to explain the sign patterns exhibited by the
coefficients of these Chebyshev quotients.

A successful approach in this direction could lead to a uniform combinatorial model for graded
multiplicities in excellent filtrations, not only for fat hook partitions but for arbitrary
partitions, thereby significantly generalizing existing results.

It is also natural to speculate that there exists a suitable $q$-analogue of the Chebyshev
polynomials such that the corresponding quotients yield formulae for the generating series of
graded multiplicities. While the present result provides such a description only in the numerical
case (i.e., upon specialization $q=1$), it would be very interesting to explore whether a
$q$-deformation of this framework can recover the full graded multiplicity generating functions.
\end{remark}
\section{The graded character of $V(\xi)$ via recursive formulas for Kostka–Foulkes polynomials} 

We now relate the recursive structure of fusion products to that of Kostka--Foulkes polynomials. This connection provides a natural framework for expressing graded multiplicities of irreducible modules in terms of cocharge Kostka--Foulkes polynomials. In particular, we will show that the graded multiplicity of an irreducible $\mathfrak{sl}_2$-module $V(r)$ in $V(\xi)$ is given by a suitably normalized cocharge Kostka--Foulkes polynomial.\\
Let $\langle\cdot,\cdot\rangle$ be the Hall inner product on $\Lambda$,
the ring of symmetric functions \cite{Macdonald}. The modified Hall--Littlewood polynomials 
$Q'_{\mu}=Q'_{\mu}(x;q)$ are defined as the basis of $\Lambda[q]$ 
dual (or adjoint) to the ordinary Hall--Littlewood polynomials 
$P_{\la}=P_{\la}(x;q)$ with respect to
$\langle\cdot,\cdot\rangle$:
\[
\langle P_{\la},Q'_{\mu}\rangle=\delta_{\la\mu}.
\]
The Kostka--Foulkes polynomials $K_{\la\mu}(q)$ are then defined as 
\[
Q'_{\mu}(x;q)=\sum_{\la} K_{\la\mu}(q)s_{\la}(x),
\]
where the $s_{\la}(x)$ are the Schur functions.
The closely related cocharge Kostka polynomials $\tilde{K}_{\la\mu}(q)$
are given by
\[
\tilde{K}_{\la\mu}(q)=q^{n(\mu)} K_{\la\mu}(1/q),
\]
where $n(\mu):=\sum_{i\geq 1} (i-1)\mu_i=\sum_{i\geq 1} \binom{\mu'_i}{2}$.

For $\mu$ a partition such that $n:=l(\mu)\geq 2$, let 
\begin{align*}
\mu^-&:=(\mu_1,\dots,\mu_{n-2},\mu_{n-1}-\mu_n,0) \\
\mu^+&:=w(\mu_1,\dots,\mu_{n-2},\mu_{n-1}+1,\mu_n-1)
\end{align*}
with $w\in S_n$ such that $\mu^+$ is a partition.

\begin{theorem}\label{thmKtilde}
Let $\la$ be a partition of length at most $2$.
Then
\begin{equation}\label{eqKthilde}
\tilde{K}_{\la\mu}(q)=\tilde{K}_{\la\mu^+}(q)
+q^{(n-1)\mu_n} \tilde{K}_{(\la_1-\mu_n,\la_2-\mu_n),\mu^-}(q).
\end{equation}
\end{theorem}

\begin{proof}
Recall that $s_{\la}(x_1,x_2)$ vanishes unless $l(\la)\leq 2$.
Hence, if in \eqref{eqKthilde} we replace $q$ by $1/q$, then
multiply both sides by $s_{\la}(x_1,x_2)$ and sum over $\la$,
we obtain
\begin{multline*}
Q'_{\mu}(x_1,x_2;q)=q^{n(\mu)-n(\mu^+)} Q'_{\mu^+}(x_1,x_2;q) \\
+q^{n(\mu)-n(\mu^-)-(n-1)\mu_n} (x_1x_2)^{\mu_n} Q'_{\mu^-}(x_1,x_2;q).
\end{multline*}
Here we have also used that
\[
s_{(\la_1+\mu_n,\la_2+\mu_n)}(x_1,x_2)
=(x_1 x_2)^{\mu_n} s_{(\la_1,\la_2)}(x_1,x_2).
\]
Since 
\begin{align*}
\mu^-_i&=\mu_i-\mu_n \chi(n-1\leq i\leq n)\\[1mm]
(\mu^+)'_i&=\mu'_i-\chi(i=\mu_n)+\chi(i=\mu_{n-1}+1),
\end{align*}
the exponents of $q$ on the right can be simplified, resulting in
\begin{multline}\label{rewrite}
Q'_{\mu}(x_1,x_2;q)\\=q^{n-1-\mu'_{\mu_{n-1}+1}} Q'_{\mu^+}(x_1,x_2;q) 
+q^{(n-2)\mu_n} (x_1x_2)^{\mu_n} Q'_{\mu^-}(x_1,x_2;q).
\end{multline}
To prove this we begin by recalling that for 
$\mu=(\mu_1,\dots,\mu_n)$ we have \cite{Jing91}
\begin{equation}\label{Jing}
Q'_{\mu}(x_1,x_2;q)=B_{\mu_1}\cdots B_{\mu_n}(1),
\end{equation}
where $B_m=B_m(x_1,x_2;q)$ is the $q$-Bernstein operator
\begin{equation}\label{qBernstein}
(B_m \,f)(x_1,x_2)=\frac{x_1^{m+1}f(qx_1,x_2)-x_2^{m+1}f(x_1,qx_2)}{x_1-x_2}
\end{equation}
acting on polynomials $f$.
Note in particular that 
\begin{equation}\label{square}
B_m \big( (x_1x_2)^k f(x_1,x_2)\big) = 
(qx_1x_2)^k  B_m \big( f(x_1,x_2)\big)
\end{equation}
and
\begin{equation}\label{commutation}
B_m B_{m+1}-q B_{m+1}B_m=0.
\end{equation}

From \eqref{Jing} and \eqref{qBernstein} it follows that 
\[
Q'_{(m)}(x_1,x_2;q)=
\frac{x_1^{m+1}-x_2^{m+1}}{x_1-x_2}=h_m(x_1,x_2),
\]
the $m$th complete symmetric function of two variables.
(More generally, $Q'_{(m)}(x;q)=h_m(x)$ for arbitrary
alphabets $x$).
It is readily checked that for $m\geq 1$
\begin{equation}\label{hrelation}
h_m(x_1,x_2)=x_1 h_{m-1}(x_1,x_2)+x_2^m.
\end{equation}
Indeed,
\[
\frac{x_1^{m+1}-x_2^{m+1}}{x_1-x_2}-\frac{x_1^{m+1}-x_1 x_2^{m}}{x_1-x_2}
=x_2^m.
\]

After these preliminaries we will prove \eqref{rewrite} by induction 
on $l(\mu)=n$.

We recall that in two variables, for a partition $(a,b)$, one has
\[
Q'_{(a,b)}(x_1,x_2;q) = (x_1 x_2)^b \, Q'_{(a-b)}(x_1,x_2;q),
\]
and for one-part partitions,
\[
Q'_{(k)}(x_1,x_2;q) = h_k(x_1,x_2).
\]

For $n=2$ we have $n-1-\mu'_{\mu_{n-1}+1}|_{n=2}=
1-\mu'_{\mu_1+1}=1-0=1$ so that the equation to be proved is
\begin{equation}\label{rewritenistwee}
Q'_{(\mu_1,\mu_2)}(x_1,x_2;q)=q\, Q'_{(\mu_1+1,\mu_2-1)}(x_1,x_2;q) 
+(x_1x_2)^{\mu_2} Q'_{(\mu_1-\mu_2)}(x_1,x_2;q).
\end{equation}
where $(\mu_1-\mu_2)$ denotes a one-part partition.
But
\begin{align*}
Q'_{(\mu_1,\mu_2)}(x_1,x_2;q)&=B_{\mu_1} Q'_{(\mu_2)}(x_1,x_2;q) \\[1mm]
&=B_{\mu_1} h_{\mu_2}(x_1,x_2) \\
&=\frac{x_1^{\mu_1+1}h_{\mu_2}(qx_1,x_2)-x_2^{\mu_2+1}h_{\mu_2}(x_1,qx_2)}
{x_1-x_2}.
\end{align*}
Since $\mu_2\geq 1$ we may apply \eqref{hrelation} as well as 
its counterpart obtained by interchanging $x_1$ and $x_2$. Hence
\begin{align*}
&Q'_{(\mu_1,\mu_2)}(x_1,x_2;q)\\
&\quad=
\frac{x_1^{\mu_1+1}\big(qx_1 h_{\mu_2-1}(qx_1,x_2)+x_2^{\mu_2}\big)
-x_2^{\mu_1+1}\big(qx_2 h_{\mu_2-1}(x_1,qx_2)+x_1^{\mu_2}\big)}
{x_1-x_2} \\[1mm]
&\quad=q\, B_{\mu_1+1} h_{\mu_2-1}(x_1,x_2)
+(x_1x_2)^{\mu_2} \, \frac{x_1^{\mu_1-\mu_2+1}-x_2^{\mu_1-\mu_2+1}}
{x_1-x_2} \\[1mm]
&\quad=q\, B_{\mu_1+1} Q'_{(\mu_2-1)}(x_1,x_2;q)
+(x_1x_2)^{\mu_2} h_{\mu_1-\mu_2}(x_1,x_2) \\[2mm]
&\quad=q\, Q'_{(\mu_1+1,\mu_2-1)}(x_1,x_2;q)
+(x_1x_2)^{\mu_2} Q_{(\mu_1-\mu_2)}(x_1,x_2;q),
\end{align*}
which settles \eqref{rewritenistwee}.

Now assume the claim is true for lengths up to $n$ and, given
$\mu=(\mu_1,\dots,\mu_n)$ (with $\mu_n>0)$ and $\mu_0\geq\mu_1$,
define 
\[
\la:=(\mu_0,\mu)=(\mu_0,\mu_1,\dots,\mu_n)
\]
and accordingly, let
\begin{equation}\label{lambdamin}
\la^-=(\mu_0,\mu^-)=
(\mu_0,\mu_1,\dots,\mu_{n-2},\mu_{n-1}-\mu_n)
\end{equation}
and
\[
\la^{+}=\sigma(\mu_0,\dots,\mu_{n-2},\mu_{n-1}+1,\mu_n-1)
\]
with $\sigma\in S_{n+1}$ such that $\la^{+}$ is a partition.
Generically, $\mu_0\geq \mu^{+}_1$ so that
\begin{subequations}
\begin{equation}\label{lambdaplusgen}
\la^{+}=(\mu_0,\mu^+).
\end{equation}
However, when $\mu_0=\mu^+_1-1$,
implying that $\mu_0=\mu_1=\mu_2=\dots=\mu_{n-1}$
and thus $\mu^+=(\mu_1+1,\mu_2,\dots,\mu_{n-1},\mu_n-1)$,
we have
\begin{equation}\label{lambdaplusspec}
\la^+=(\mu_0+1,\mu_1,\dots,\mu_{n-1},\mu_n-1).
\end{equation}
\end{subequations}
By \eqref{Jing} and induction
\begin{align*}
Q'_{\la}&(x_1,x_2;q) \\[1mm]
&=B_{\mu_0} Q'_{\mu}(x_1,x_2;q) \\
&=q^{n-1-\mu'_{\mu_{n-1}+1}} B_{\mu_0} Q'_{\mu^+}(x_1,x_2;q)
+B_{\mu_0} \Big(q^{(n-2)\mu_n} (x_1x_2)^{\mu_n} Q'_{\mu^{-}}(x_1,x_2;q)\Big).
\end{align*}
The second term on the right can be simplified by \eqref{square}
and then \eqref{lambdamin} to 
\begin{align*}
q^{(n-2)\mu_n} (qx_1x_2)^{\mu_n} B_{\mu_0} Q'_{\mu^{-}}(x_1,x_2;q)
&=q^{(n-1)\mu_n} (x_1x_2)^{\mu_n} Q'_{\la^{-}}(x_1,x_2;q) \\
&=q^{(n-1)\la_{n+1}} (x_1x_2)^{\la_{n+1}} Q'_{\la^{-}}(x_1,x_2;q).
\end{align*}
For the first term on the right we need to distinguish
\eqref{lambdaplusgen} and \eqref{lambdaplusspec}.
In the case of \eqref{lambdaplusgen},
$\la'_{\la_n+1}=\mu'_{\mu_{n-1}+1}+1$, resulting in
\[
q^{n-1-\mu'_{\mu_{n-1}+1}} B_{\mu_0} Q'_{\mu^+}(x_1,x_2;q)=
q^{n-\la'_{\la_n+1}} Q'_{\la^+}(x_1,x_2;q).
\]
In the case of \eqref{lambdaplusspec}, since $\mu_0=\mu_1$
and $\mu^+=(\mu_1+1,\mu_2,\dots,\mu_{n-1},\mu_n-1)$,
we can use \eqref{commutation} to find
\begin{align*}
B_{\mu_0} Q'_{\mu^+}(x_1,x_2;q)&=
B_{\mu_0} Q'_{(\mu_1+1,\mu_2,\dots,\mu_{n-1},\mu_n-1)}(x_1,x_2;q) \\
&=B_{\mu_0} B_{\mu_1+1} Q'_{(\mu_2,\dots,\mu_{n-1},\mu_n-1)}(x_1,x_2;q) \\
&=qB_{\mu_0+1} B_{\mu_1} Q'_{(\mu_2,\dots,\mu_{n-1},\mu_n-1)}(x_1,x_2;q) \\
&=q\,Q'_{(\mu_0+1,\mu_1,\mu_2,\dots,\mu_{n-1},\mu_n-1)}(x_1,x_2;q) \\
&=q\,Q'_{\la^+}(x_1,x_2;q).
\end{align*}
Moreover, since $\mu_1=\dots=\mu_{n-1}$ and $\la_1=\dots=\la_n$,
\[
\mu'_{\mu_{n-1}+1}=\mu'_{\mu_1+1}=0=\la'_{\la_1+1}=\la'_{\la_n+1}
\]
so that $q^{n-1-\mu'_{\mu_{n-1}+1}}=q^{n-1-\la'_{\la_n+1}}$.
Therefore, we once again find
\[
q^{n-1-\mu'_{\mu_{n-1}+1}} B_{\mu_0} Q'_{\mu^+}(x_1,x_2;q)
=q^{n-\la'_{\la_n+1}} Q'_{\la^+}(x_1,x_2;q).
\]
In summary, for $\la=(\mu_0,\mu)$ of length $n+1$,
\[
Q'_{\la}(x_1,x_2;q)
=q^{n-\la'_{\la_n+1}} Q'_{\la^+}(x_1,x_2;q)+
q^{(n-1)\la_{n+1}} (x_1x_2)^{\la_{n+1}} Q'_{\la^{-}}(x_1,x_2;q),
\]
completing the proof.
\end{proof}
\begin{remark}
We note that alternative recursive formulas for Kostka--Foulkes polynomials have
been studied in the literature; see, for example, Bryan and Jing \cite{BJ}
for an iterative formula. The recursion established in Theorem 3.1 is of a
different nature and is tailored to the structure of fusion products considered
in this paper.
\end{remark}

The following structural result on fusion products will play a key role in relating
the recursion for graded multiplicities to that of Kostka--Foulkes polynomials.
\begin{theorem}[{\cite{CV}}]\label{ses}
Let $\xi = (\xi_1, \dots, \xi_s > 0)$ be a partition with $s$ parts. For $s>1$, there exists a short exact sequence of $\mathfrak{sl}_2[t]$-modules:
\[
0 \to \tau_{(s-1)\xi_s} V(\xi^{-}) 
\to V(\xi) 
\to V(\xi^{+}) \to 0.
\] 
\end{theorem}

\begin{theorem}
If we denote by $[V(\xi) : V(r)]_q$ the graded multiplicity of the irreducible $\mathfrak{sl}_2$-module $V(r)$ in $V(\xi)$, then
\[
[V(\xi) : V(r)]_q = \tilde{K}_{\mu\xi}(q)
\]
where $\mu = \left( \frac{|\xi|+r}{2}, \frac{|\xi|-r}{2} \right)$.
\end{theorem}

\begin{proof}
We define a partial order on the set of partitions as follows: given partitions $\xi$ and $\xi'$ with $s$ and $s'$ parts respectively, we say $\xi > \xi'$ if $s > s'$, and if $s = s'$, then $\xi_s > \xi'_s$.

We now prove the theorem by induction. If $\xi = (r)$, then the result is immediate since $V(\xi) \cong V(r)$.

Assume the statement holds for all partitions with at most $s-1$ parts, and let $\xi$ be a partition with $s$ parts. Then the inductive hypothesis applies to both $V(\xi^-)$ and $V(\xi^+)$.

From the short exact sequence in Theorem \ref{ses}, we obtain the recursive formula
\[
[V(\xi) : V(r)]_q = [V(\xi^+) : V(r)]_q + q^{(s-1)\xi_s} [V(\xi^-) : V(r)]_q.
\]

The result now follows from this recursion together with Theorem \ref{thmKtilde}.
\end{proof}

\begin{cor}
\[
\mathrm{Ch}_{\mathrm{gr}} V(\xi)
=
\sum_{n:\,|\xi|-n \in 2\Z_+}
[V(\xi):V(n)]_q \, \mathrm{Ch}\,V(n)
=
\sum_{\mu}
\tilde{K}_{\mu\xi}(q) S_\mu.
\]
\end{cor}

\section*{Acknowledgements}
The author gratefully acknowledges Ole Warnaar for communicating the proof of the recursion for cocharge Kostka polynomials. The author thanks Travis Scrimshaw and Katsuyuki Naoi for a careful reading of the manuscript and for providing valuable comments. The author also expresses sincere gratitude to the referee for a thorough review and for suggesting helpful corrections. Financial support from the National Institute of Science Education and Research (NISER), Bhubaneswar, and the Homi Bhabha National Institute (HBNI), Mumbai, is gratefully acknowledged.

\end{document}